\documentclass{amsart}%
\usepackage{amsfonts}
\usepackage{amsmath}
\usepackage{amssymb}
\usepackage{graphicx}%
\providecommand{\U}[1]{\protect\rule{.1in}{.1in}}
\newtheorem{theorem}{Theorem}
\theoremstyle{plain}

\newtheorem{corollary}{Corollary}

\newtheorem{definition}{Definition}
\newtheorem{example}{Example}

\newtheorem{proposition}{Proposition}
\newtheorem{remark}{Remark}

\numberwithin{equation}{section}
\begin{document}
\title[$S$-$n$-ideals of commutative rings]{$S$-$n$-ideals of commutative rings}
\author{Hani Khashan}
\address{Department of Mathematics, Faculty of Science, Al al-Bayt University, Al
Mafraq, Jordan.}
\email{hakhashan@aabu.edu.jo.}
\author{Ece Yetkin Celikel}
\address{Electrical-Electronics Engineering, Faculty of Engineering, HasanKalyoncu
University, Gaziantep, Turkey.}
\email{ece.celikel@hku.edu.tr, yetkinece@gmail.com.}
\thanks{This paper is in final form and no version of it will be submitted for
publication elsewhere.}

\begin{abstract}
Let $R$ be a commutative ring with identity and $S$ a multiplicatively closed
subset of $R$. This paper aims to introduce the concept of $S$-$n$-ideals as a
generalization of $n$-ideals. An ideal $I$ of $R$ disjoint with $S$ is called
an $S$-$n$-ideal if there exists $s\in S$ such that whenever $ab\in I$ for
$a,b\in R$, then $sa\in\sqrt{0}$ or $sb\in I$. The relationship among $S$%
-$n$-ideals, $n$-ideals, $S$-prime and $S$-primary ideals are clarified.
Several properties, characterizations and examples are presented such as
investigating $S$-$n$-ideals under various contexts of constructions including
direct products, localizations and homomorphic images. Furthermore, for $m\in%
\mathbb{N}
$ and some particular $S$, all $S$-$n$-ideals of the ring $%
\mathbb{Z}
_{m}$ are completely determined. The last section is devoted for studying the
$S$-$n$-ideals of the idealization ring and amalgamated algebra.

\end{abstract}
\maketitle

\section{Introduction}

Throughout this paper, we assume that all rings are commutative with non-zero
identity. For a ring $R$, we will denote by $U(R),$ $reg(R)$ and $Z(R)$, the
set of unit elements, regular elements and zero-divisor elements of $R$,
respectively. For an ideal $I$ of $R$, the radical of $I$ denoted by $\sqrt
{I}$ is the ideal $\{a\in R:a^{n}\in I$ for some positive integer $n\}$ of
$R$. In particular, $\sqrt{0}$ denotes the set of all nilpotent elements of
$R$. We recall that a proper ideal $I$ of a ring $R$ is called prime (primary)
if for $a,b\in R$, $\ ab\in I$ implies $a\in I$ or $b\in I$ ($b\in\sqrt{I}$).
Several generalizations of prime and primary ideals were introduced and
studied, (see for example \cite{AB}-\cite{wp}, \cite{Darani}, \cite{Yas}).

Let $S$ be a multiplicatively closed subset of a ring $R$ and $I$ an ideal of
$R$ disjoint with $S$. Recently, Hamed and Malek \cite{S-prime} used a new
approach to generalize prime ideals by defining $S$-prime ideals. $I$ is
called an $S$-prime ideal of $R$ if there exists an $s\in S$ such that for all
$a,b\in R$ whenever $ab\in I$, then $sa\in I$ or $sb\in I$. Then analogously,
Visweswaran \cite{S-primary} introduced the notion of $S$-primary ideals. $I$
is called $S$-primary ideal of $R$ if there exists an $s\in S$ such that for
all $a,b\in R$ if $ab\in I$, then $sa\in I$ or $sb\in\sqrt{I}$. Many other
generalizations of $S$-prime and $S$-primary ideals have been studied. For
example, in \cite{WS-prime}, the authors defined $I$ to be a weakly $S$-prime
ideal if there exists an $s\in S$ such that for all $a,b\in R$ if $0\neq ab\in
I$, then $sa\in I$ or $sb\in I$.

In 2015, Mohamadian \cite{M} defined a new type of ideals called $r$-ideals.
An ideal $I$ of a ring $R$ is said to be $r$ -ideal, if $ab\in I$ and $a\notin
Z(R)$ imply that $b\in I$ for each $a,b\in R$. Generalizing this concept, in
2017 the notion of $n$-ideals was first introduced and studied \cite{Tekir}.
The authors called a proper ideal $I$ of $R$ an $n$-ideal if $ab\in I$ and
$a\notin\sqrt{0}$ imply that $b\in I$ for each $a,b\in R$. Many other
generalizations of $n$-ideals have been introduced recently, see for example
\cite{Hani} and \cite{Ece}.

Motivated from these studies, in this article, we determine the structure of
$S$-$n$-ideals of a ring. We call $I$ an $S$-$n$-ideal of a ring $R$ if there
exists an (fixed) $s\in S$ such that for all $a,b\in R$ if $ab\in I$ and
$sa\notin\sqrt{0}$, then $sb\in I$. We call this fixed element $s\in S$ an
$S$-element of $I$. Clearly, for any $S$, every $n$-ideal is an $S$-$n$-ideal
and the classes of $n$-ideals and $S$-$n$-ideals coincide if $S\subseteq
U(R)$. In Section 2, we start by giving an example of an $S$-$n$-ideal of a
ring $R$ that is not an $n$-ideal. Then we give many properties of $S$%
-$n$-ideals and show that $S$-$n$-ideals enjoy analogs of many of the
properties of $n$-ideals. Also we discuss the relationship among $S$%
-$n$-ideals, $n$-ideals, $S$-prime and $S$-primary ideals, (Propositions
\ref{p1}, \ref{max} and Examples \ref{e1}, \ref{e2}). In Theorems \ref{char}
and \ref{char2}, we present some characterizations for $S$-$n$-ideals of a
general commutative ring. Moreover, we investigate some conditions under which
$(I:_{R}s)$ is an $S$-$n$-ideal of $R$ for an $S$-$n$-ideal $I$ of $R$ and an
$S$-element $s$ of $I$, (Propositions \ref{(I:s)}, \ref{reg} and Example
\ref{e3}). For a particular case that $S\subseteq reg(R)$, we justify some
other results. For example, in this case, we prove that a maximal $S$%
-$n$-ideal of $R$ is $S$-prime, (Proposition \ref{max}). In addition, we show
in Proposition \ref{Every} that every proper ideal of a ring $R$ is an $S$%
-$n$-ideal if and only if $R$ is a UN-ring (a ring for which every nonunit
element is a product of a unit and a nilpotent).

Let $n\in%
\mathbb{N}
$, say, $n=p_{1}^{r_{1}}p_{2}^{r_{2}}...p_{k}^{r_{k}}$ where $p_{1}%
,p_{2},...,p_{k}$ are distinct prime integers and $r_{i}\geq1$ for all $i$.
Then for all $2\leq i\leq k-1$, $S_{p_{1}p_{2}...p_{i-1}p_{i+1}...p_{k}%
}=\left\{  \bar{p}_{1}^{m_{1}}\bar{p}_{2}^{m_{2}}...\bar{p}_{i-1}^{m_{i-1}%
}\bar{p}_{i+1}^{m_{i+1}}...\bar{p}_{k-1}^{m_{k-1}}:m_{j}\in%
\mathbb{N}
\cup\left\{  0\right\}  \right\}  $ is a multiplicatively closed subset of $%
\mathbb{Z}
_{n}$. In Theorem \ref{Zn,gen}, we determine all $S_{p_{1}p_{2}...p_{i-1}%
p_{i+1}...p_{k}}$-$n$-ideals of $%
\mathbb{Z}
_{n}$ for all $i$. In particular, we determine all $S_{p}$-$n$-ideals of $%
\mathbb{Z}
_{n}$ where $S_{p}=\left\{  1,\bar{p},\bar{p}^{2},\bar{p}^{3},...\right\}  $
for any $p$ prime integer dividing $n$, (Theorem \ref{Zn}). Furthermore, we
study the stability of $S$-$n$-ideals with respect to various ring theoretic
constructions such as localization, factor rings and direct product of rings,
(Propositions \ref{Inverse g}, \ref{f}, \ref{quot} and \ref{cart}).

Let $R$ be a ring and $M$ be an $R$-module. For a multiplicatively closed
subset $S$ of $R$, the set $S(+)M=\left\{  (s,m):s\in S\text{, }m\in
M\right\}  $ is clearly a multiplicatively closed subset of the idealization
ring $R(+)M$. In Section 3, first, we clarify the relation between the $S$%
-$n$-ideals of a ring $R$ and the $S(+)M$-$n$-ideals $R(+)M$, (Proposition
\ref{Id}). For rings $R$ and $R^{\prime}$, an ideal $J$ of $R^{\prime}$ and a
ring homomorphism $f:R\rightarrow R^{\prime}$, the amalgamation of $R$ and
$R^{\prime}$ along $J$ with respect to $f$ is the subring $R\Join
^{f}J=\left\{  (r,f(r)+j):r\in R\text{, }j\in J\right\}  $ of $R\times
R^{\prime}$. Clearly, the set $S\Join^{f}J=\left\{  (s,f(s)+j):s\in S\text{,
}j\in J\right\}  $ is a multiplicatively closed subset of $R\Join^{f}J$
whenever $S$ is a multiplicatively closed subset of $R$. We finally determine
when the ideals $I\Join^{f}J=\left\{  (i,f(i)+j):i\in I\text{, }j\in
J\right\}  $ and $\bar{K}^{f}=\{(a,f(a)+j):a\in R$, $j\in J$, $f(a)+j\in K\}$
of $R\Join^{f}J$ are $(S\Join^{f}J)$-$n$-ideals, (Theorems \ref{ama} and
\ref{ama2}).

\section{Properties of $S$-$n$-ideals}

\begin{definition}
Let $R$ be a ring, $S$ be a multiplicatively closed subset of $R$ and $I$ be
an ideal of $R$ disjoint with $S$. We call $I$ an $S$-$n$-ideal of $R$ if
there exists an (fixed) $s\in S$ such that for all $a,b\in R$ if $ab\in I$ and
$sa\notin\sqrt{0}$, then $sb\in I$. This fixed element $s\in S$ is called an
$S$-element of $I.$
\end{definition}

Let $I$ be an ideal of a ring $R$. If $I$ is an $n$-ideal of $R$, then clearly
$I$ is an $S$-$n$-ideal for any multiplicatively closed subset of $R$ disjoint
with $I$. However, it is clear that the classes of $n$-ideals and $S$%
-$n$-ideals coincide if $S\subseteq U(R)$. Moreover, obviously any $S$%
-$n$-ideal is an $S$-primary ideal and the two concepts coincide if the ideal
is contained in $\sqrt{0}$. However, the converses of these implications are
not true in general as we can see in the following examples.

\begin{example}
\label{e1}Let $R=%
\mathbb{Z}
_{12}$, $S=\{\overline{1},\overline{3},\overline{9}\}$ and consider the ideal
$I=<\overline{4}>$. Choose $s=\overline{3}\in S$ and let $a,b\in R$ with
$ab\in I$ but $3b\notin I$. Now, $ab\in<\overline{2}>$ implies $a\in
<\overline{2}>$ or $b\in<\overline{2}>$. Assume that $a\notin<\overline{2}>$
and $b\in<\overline{2}>$. Since $a\notin<\overline{2}>$, then $a\in
\{\overline{1},\overline{3},\overline{5},\overline{7},\overline{9}%
,\overline{11}\}$ and since $3b\notin I$, we have $b\in\{\overline
{2},\overline{6},\overline{10}\}$. Thus, in each case $ab\notin I$, a
contradiction. Hence, we must have $a\in<\overline{2}>$ and so $\overline
{3}a\in<\overline{6}>=\sqrt{0}.$ On the other hand, $I$ is not an $n$-ideal as
$\overline{2}\cdot\overline{2}\in I$ but neither $\overline{2}\in\sqrt{0}$ nor
$\overline{2}\in I.$
\end{example}

A (prime) primary ideal of a ring $R$ that is not an $n$-ideal is a direct
example of an ($S$-prime) $S$-primary ideal that is not an $S$-$n$-ideal where
$S=\left\{  1\right\}  $. For a less trivial example, we have the following.

\begin{example}
\label{e2}Let $R=%
\mathbb{Z}
\lbrack X]$ and let $I=\left\langle 4x\right\rangle $. consider the
multiplicatively closed subset $S=\{4^{m}:m\in%
\mathbb{N}
\cup\{0\}\}$ of $R$. Then $I$ is an $S$-prime (and so $S$-primary) ideal of
$R$, \cite[Example 2.3]{S-primary}. However, $I$ is not an $S$-$n$-ideal since
for all $s=4^{m}\in S$, we have $(2x)(2)\in I$ but $s(2x)\notin\sqrt{0_{%
\mathbb{Z}
\lbrack x]}}$ and $s(2)\notin I$.
\end{example}

\begin{proposition}
\label{p1}Let $S$ be a multiplicatively closed subset of a ring $R$ and $I$ be
an ideal of $R$ disjoint with $S$.
\end{proposition}

\begin{enumerate}
\item If $I$ be an $S$-$n$-ideal, then $sI\subseteq\sqrt{0}$ for some $s\in
S$. If moreover, $S\subseteq reg(R)$, then $I\subseteq\sqrt{0}$.

\item $\sqrt{0}$ is an $S$-$n$-ideal of $R$ if and only if $\sqrt{0}$ is an
$S$-prime ideal of $R$.

\item Let $S\subseteq reg(R).$ Then $0$ is an $S$-$n$-ideal of $R$ if and only
if $0$ is an $n$-ideal.
\end{enumerate}

\begin{proof}
(1) Let $a\in I$. Since $I\cap S=\emptyset$, $s\cdot1\notin I$ for all $s\in
S.$ Hence, $a\cdot1\in I$ implies that there exists an $s\in S$ such that
$sa\in\sqrt{0}$. Thus, $sI\subseteq\sqrt{0}$ as desired. Moreover, if
$S\subseteq reg(R)$, then clearly $I\subseteq\sqrt{0}$.

(2) Clear.

(3) Suppose $s$ be an $S$-element of $0$ and $ab=0$ for some $a,b\in R$. Then
$sa\in\sqrt{0}$ or $sb=0$ which implies $s^{n}a^{n}=0$ for some positive
integer $n$ or $sb=0$. Since $S\subseteq reg(R),$ we have $a^{n}=0$ or $b=0$,
as needed.
\end{proof}

Next, we characterize $S$-$n$-ideals of rings by the following.

\begin{theorem}
\label{char}Let $S$ be a multiplicatively closed subset of a ring $R$ and $I$
be an ideal of $R$ disjoint with $S$. The following statements are equivalent.
\end{theorem}

\begin{enumerate}
\item $I$ is an $S$-$n$-ideal of $R$.

\item There exists an $s\in S$ such that for any two ideals $J,K$ of $R$, if
$JK\subseteq I$, then $sJ\subseteq\sqrt{0}$ or $sK\subseteq I.$
\end{enumerate}

\begin{proof}
(1)$\Rightarrow$(2). Suppose $I$ is an $S$-$n$-ideal of $R$. Assume on the
contrary that for each $s\in S$, there exist two ideals $J^{\prime},K^{\prime
}$ of $R$ such that $J^{\prime}K^{\prime}\subseteq I$ but $sJ^{\prime
}\nsubseteq\sqrt{0}$ and $sK^{\prime}\nsubseteq I$. Then, for each $s\in S$,
we can find two elements $a\in J^{\prime}$ and $b\in K^{\prime}$ such that
$ab\in I$ but neither $sa\in\sqrt{0}$ nor $sb\in I$. By this contradiction, we
are done.

(2)$\Rightarrow$(1). Let $a,b\in R$ with $ab\in I$. Taking $J=<a>$ and $K=<b>$
in (2), we get the result.
\end{proof}

\begin{theorem}
\label{char2}Let $S$ be a multiplicatively closed subset of a ring $R$ and $I$
be an ideal of $R$ disjoint with $S$. If $\sqrt{0}$ is an $S$-$n$-ideal of
$R$, then the following are equivalent.
\end{theorem}

\begin{enumerate}
\item $I$ is an $S$-$n$-ideal of $R$.

\item There exists $s\in S$ such that for ideals $I_{1},I_{2},...,I_{n}$ of
$R$, if $I_{1}I_{2}\cdots I_{n}\subseteq I$, then $sI_{j}\subseteq\sqrt{0}$ or
$sI_{k}\subseteq I$ for some $j,k\in\{1,...,n\}$.

\item There exists $s\in S$ such that for elements $a_{1},a_{2},...,a_{n}$ of
$R$, if $a_{1}a_{2}\cdots a_{n}\in I$, then $sa_{j}\in\sqrt{0}$ or $sa_{k}\in
I$ for some $j,k\in\{1,...,n\}.$
\end{enumerate}

\begin{proof}
(1)$\Rightarrow$(2). Let $s_{1}\in S$ be an $S$-element of $I$. To prove the
claim, we use mathematical induction on $n$. If $n=2$, then the result is
clear by Theorem \ref{char}. Suppose $n\geq3$ and the claim holds for $n-1.$
Let $I_{1},I_{2},$..., $I_{n}$ be ideals of $R$ with $I_{1}I_{2}\cdots
I_{n}\subseteq I$. Then by Theorem \ref{char}, we conclude that either
$s_{1}I_{1}\subseteq\sqrt{0}$ or $s_{1}I_{2}\cdots I_{n}\subseteq I$. Assume
$(s_{1}I_{2})\cdots I_{n}\subseteq I.$ By the induction hypothesis, we have
either, say, $s_{1}^{2}I_{2}\subseteq\sqrt{0}$ or $s_{1}I_{k}\subseteq I$ for
some $k\in\{3,...,n\}$. Assume $s_{1}^{2}I_{2}\subseteq\sqrt{0}$ and choose an
$S$-element $s_{2}\in S$ of $\sqrt{0\text{.}}$ If $s_{2}(s_{1}^{2}%
R)\subseteq\sqrt{0}\cap S$, we get a contradiction. Thus, $s_{2}I_{2}%
\subseteq\sqrt{0}$. By choosing $s=s_{1}s_{2}$, we get $sI_{j}\subseteq
\sqrt{0}$ or $sI_{k}\subseteq I$ for some $j,k\in\{1,...,n\}$, as needed.

(2)$\Rightarrow$(3). This is a particular case of (2) by taking $I_{j}%
:=<a_{j}>$ for all $j\in\{1,...,n\}$.

(3)$\Rightarrow$(1). Clear by choosing $n=2$ in (3).
\end{proof}

\begin{proposition}
\label{(I:s)}Let $S$ be a multiplicatively closed subset of a ring $R$ and $I$
be an ideal of $R$ disjoint with $S$. Then
\end{proposition}

\begin{enumerate}
\item If $(I:s)$ is an $n$-ideal of $R$ for some $s\in S$, then $I$ is an
$S$-$n$-ideal.

\item If $I$ is an $S$-$n$-ideal and $(\sqrt{0}:s)$ is an $n$-ideal where
$s\in S$ is an $S$-element of $I$, then $(I:s)$ is an $n$-ideal of $R$.

\item If $I$ is an $S$-$n$-ideal and $S\subseteq reg(R)$, then $(I:s)$ is an
$n$-ideal of $R$ for any $S$-element $s$ of $I$.
\end{enumerate}

\begin{proof}
(1) Suppose that $(I:s)$ is an $n$-ideal of $R$ for some $s\in S$. We show
that $s$ is an $S$-element of $I$. Let $a,b\in R$ with $ab\in I$ and
$sa\notin\sqrt{0}.$ Then $ab\in(I:s)$ and $a\notin\sqrt{0}$ imply that
$b\in(I:s)$.Thus, $sb\in I$ and $I$ is an $S$-$n$-ideal.

(2) Suppose $a,b\in R$ with $ab\in(I:s)$. Then $a(sb)\in I$ which implies
$sa\in\sqrt{0}$ or $s^{2}b\in I$. Suppose $sa\in\sqrt{0}$. Since $(\sqrt
{0}:s)$ is an $n$-ideal, $(\sqrt{0}:s)=\sqrt{0}$ by \cite[Proposition
2.3]{Tekir} and so $a\in\sqrt{0}$. Now, suppose $s^{2}b\in I.$ If $sb\notin
I$, then since $I$ is an $S$-$n$-ideal, $s^{3}\in\sqrt{0}$ and so $s\in
\sqrt{0}$ which contradicts the assumption that $(\sqrt{0}:s)$ is proper.
Thus, $sb\in I$ and $b\in(I:s)$ as needed.

(3) Suppose $S\subseteq reg(R)$ and $I$ is an $S$-$n$-ideal. Let $a,b\in R$
with $ab\in(I:s)$ so that $a(sb)\in I$ . If $sa\in\sqrt{0}$, then $s^{m}%
a^{m}=0$ for some integer $m$. Since $S\subseteq reg(R)$, we get $a^{m}=0$ and
so $a\in\sqrt{0}$. If $s^{2}b\in I$, then similar to the proof of (2) we
conclude that $b\in(I:s)$.
\end{proof}

Note that the conditions that $(\sqrt{0}:s)$ is an $n$-ideal in (2) and
$S\subseteq reg(R)$ in (3) of Proposition \ref{(I:s)} are crucial. Indeed,
consider $R=%
\mathbb{Z}
_{12}$, $S=\{\overline{1},\overline{3},\overline{9}\}$. We showed in Example
\ref{e1} that $I=<\overline{4}>$ is an $S$-$n$-ideal which is not an
$n$-ideal, and so $(I:\overline{3})=I$ is not an $n$-ideal. Here, observe that
$S\nsubseteq reg(R)$ and $(\sqrt{0}:3)=<\overline{2}>$ is not an $n$-ideal of
$%
\mathbb{Z}
_{12}$.

\begin{proposition}
\label{reg}Let $S\subseteq reg(R)$ be a multiplicatively closed subset of a
ring $R$ and $I$ be an $S$-prime ideal of $R$. Then $I$ is an $S$-$n$-ideal if
and only if $(I:s)=\sqrt{0}$ for some $s\in S$.
\end{proposition}

\begin{proof}
Suppose $I$ is an $S$-$n$-ideal of $R$ and $s_{1}$ be an $S$-element of $I$.
Then $(I:s_{1})$ is an $n$-ideal of $R$ by Proposition \ref{(I:s)}. Moreover,
$(I:ts_{1})$ is an $n$-ideal for all $t\in S$. Indeed, if $ab\in(I:ts_{1})$
for $a,b\in R$, then $abts_{1}\in I$ and so either $s_{1}^{2}a\in\sqrt{0}$ or
$s_{1}tb\in I$. If $s_{1}^{2}a\in\sqrt{0}$, then $a\in\sqrt{0}$ as $S\subseteq
reg(R)$. Otherwise, we have $b\in(I:ts_{1})$ as needed. Since $I$ is an
$S$-prime ideal of $R$, $(I:s_{2})$ is a prime ideal of $R$ where $s_{2}\in S$
such that whenever $ab\in I$\ for $a,b\in R$, either $s_{2}a\in I$ or
$s_{2}b\in I$, \cite[Proposition 1]{S-prime}. Similar to the above argument,
we can also conclude that $(I:ts_{2})$ is a prime ideal for all $t\in S$. Now,
choose $s=s_{1}s_{2}$. Then $(I:s)$ is both a prime and an $n$-ideal of $R$
and so $(I:s)=\sqrt{0}$ by \cite[Proposition 2.8]{Tekir}. Conversely, suppose
$(I:s)=\sqrt{0}$ for some $s\in S$. Since $I$ is an $S$-prime ideal, then
$(I:s^{\prime})$ is a prime ideal of $R$ for some $s^{\prime}\in S$. Moreover,
if $a\in(I:s^{\prime})$, then $as^{\prime}\in I\subseteq(I:s)\subseteq\sqrt
{0}$ and so $a\in\sqrt{0}$ as $S\subseteq reg(R)$. Thus, $(I:s^{\prime}%
)=\sqrt{0}$ is a prime ideal and so it an $n$-ideal again by \cite[Proposition
2.8]{Tekir}. Therefore, $I$ is an $S$-$n$-ideal by Proposition \ref{(I:s)}.
\end{proof}

In the following example we justify that the condition $S\subseteq reg(R)$ can
not be omitted in Proposition \ref{reg}.

\begin{example}
\label{e3}The ideal $I=<\overline{2}>$ of $%
\mathbb{Z}
_{12}$ is prime and so $S$-prime for $S=\{\overline{1},\overline{3}%
,\overline{9}\}\nsubseteq reg(%
\mathbb{Z}
_{12})$. Moreover, one can directly see that $s=3$ is an $S$-element of $I$
and so $I$ is also an $S$-$n$-ideal of $%
\mathbb{Z}
_{12}$. But $(I:s)=I\neq\sqrt{0}$ for all $s\in S$.
\end{example}

A ring $R$ is said to be a UN-ring if every nonunit element is a product of a
unit and a nilpotent. Next, we obtain a characterization for rings in which
every proper ideal is an $S$-$n$-ideal where $S\subseteq reg(R).$

\begin{proposition}
\label{Every}Let $S\subseteq reg(R)$ be a multiplicatively closed subset of a
ring $R$. The following are equivalent.
\end{proposition}

\begin{enumerate}
\item Every proper ideal of $R$ is an $n$-ideal.

\item Every proper ideal of $R$ is an $S$-$n$-ideal.

\item $R$ is a UN-ring.
\end{enumerate}

\begin{proof}
Since (1)$\Rightarrow$(2) is straightforward and (3)$\Rightarrow$(1) is clear
by \cite[Proposition 2.25]{Tekir}, we only need to prove (2)$\Rightarrow$(3).

(2)$\Rightarrow$(3). Let $I$ be a prime ideal of $R$. Then $I$ is an $S$-prime
and from our assumption, it is also an $S$-$n$-ideal. Thus $I\subseteq
(I:s)=\sqrt{0}$ is a prime ideal of $R$ by Proposition \ref{reg}. Thus
$\sqrt{0}$ is the unique prime ideal of $R$ and so $R$ is a UN-ring by
\cite[Proposition 2 (3)]{C}.
\end{proof}

The equivalence of (1) and (2) In Proposition \ref{Every} need not be true
if\ $S\nsubseteq reg(R)$.

\begin{example}
Consider the ring $%
\mathbb{Z}
_{6}$ and let $S=\left\{  1,3\right\}  $. If $I=\left\langle \bar
{0}\right\rangle $ or $\left\langle \bar{2}\right\rangle $, then a simple
computations can show that $I$ is an $S$-$n$-ideal of $%
\mathbb{Z}
_{6}$. However, $%
\mathbb{Z}
_{6}$ has no proper $n$-ideals, \cite[Example 2.2]{Tekir}.
\end{example}

A ring $R$ is said to be von Neumann regular if for all $a\in R$, there exists
an element $b\in R$ such that $a=a^{2}b$.

\begin{proposition}
\label{integ}Let $S\subseteq reg(R)$ be a multiplicatively closed subset of a
ring $R.$
\end{proposition}

\begin{enumerate}
\item Let $R$ be a reduced ring. Then $R$ is an integral domain if and only if
there exists an $S$-prime ideal$~$of $R$ which is also an $S$-$n$-ideal

\item $R$ is a field if and only if $R$ is von Neumann regular and $0$ is an
$S$-$n$-ideal of $R$.
\end{enumerate}

\begin{proof}
(1) Let $R$ be an integral domain. Since $0=\sqrt{0}$ is prime, it is also an
$n$-ideal again by \cite[Corollary 2.9]{Tekir}. Thus $\sqrt{0}$ is both
$S$-prime and $S$-$n$-ideal of $R$, as required. Conversely, suppose $I$ is
both $S$-prime and $S$-$n$-ideal of $R$. Hence, from Proposition \ref{reg} we
conclude $(I:s)=\sqrt{0}$ which is an $n$-ideal by Proposition \ref{(I:s)}.
$\sqrt{0}=0$ is also a prime ideal by \cite[Corollary 2.9]{Tekir}, and thus
$R$ is an integral domain.

(2) Since $S\subseteq reg(R),$ from Proposition \ref{p1}, $0$ is an $S$%
-$n$-ideal of $R$ if and only if $0$ is an $n$-ideal. Thus, the claim is clear
by \cite[Theorem 2.15]{Tekir}.
\end{proof}

Let $n\in%
\mathbb{N}
$. For any prime $p$ dividing $n$, we denote the multiplicatively closed
subset $\left\{  1,\bar{p},\bar{p}^{2},\bar{p}^{3},...\right\}  $ of $%
\mathbb{Z}
_{n}$ by $S_{p}$. Next, for any $p$ dividing $n$, we clarify all $S_{p}$%
-$n$-ideals of $%
\mathbb{Z}
_{n}$.

\begin{theorem}
\label{Zn}Let $n\in%
\mathbb{N}
$.
\end{theorem}

\begin{enumerate}
\item If $n=p^{r}$ for some prime integer $p$ and $r\geq1$, then $%
\mathbb{Z}
_{n}$ has no $S_{p}$-$n$-ideals.

\item If $n=p_{1}^{r_{1}}p_{2}^{r_{2}}$ where $p_{1}$ and $p_{2}$ are distinct
prime integers and $r_{1},r_{2}\geq1$, then for all $i=1,2$, every ideal of $%
\mathbb{Z}
_{n}$ disjoint with $S_{p_{i}}$ is an $S_{p_{i}}$-$n$-ideal.

\item If $n=p_{1}^{r_{1}}p_{2}^{r_{2}}...p_{k}^{r_{k}}$ where $p_{1}%
,p_{2},...,p_{k}$ are distinct prime integers and $k\geq3$, then for all
$i=1,2,...,k$, $%
\mathbb{Z}
_{n}$ has no $S_{p_{i}}$-$n$-ideals.
\end{enumerate}

\begin{proof}
(1) Clear since $I\cap S_{p}\neq\phi$ for any ideal $I$ of $%
\mathbb{Z}
_{n}$.

(2) Let $I=\left\langle \bar{p}_{1}^{t_{1}}\bar{p}_{2}^{t_{2}}\right\rangle $
be an ideal of $%
\mathbb{Z}
_{n}$ distinct with $S_{p_{1}}$. Then we must have $t_{2}\geq1$. Choose
$s=\bar{p}_{1}^{t_{1}}\in S_{p_{1}}$ and let $ab\in I$ for $a,b\in%
\mathbb{Z}
_{n}$. If $a\in\left\langle \bar{p}_{2}\right\rangle $, then $sa\in
\left\langle \bar{p}_{1}\bar{p}_{2}\right\rangle =\sqrt{0}$. If $a\notin
\left\langle \bar{p}_{2}\right\rangle $, then clearly $b\in\left\langle
\bar{p}_{2}^{t_{2}}\right\rangle $ and so $sb\in I$. Therefore, $I$ is an
$S_{p_{1}}$-$n$-ideal of $%
\mathbb{Z}
_{n}$. By a similar argument, we can show that every ideal of $%
\mathbb{Z}
_{n}$ distinct with $S_{p_{2}}$ is an $S_{p_{2}}$-$n$-ideal.

(3) Let $I=\left\langle \bar{p}_{1}^{t_{1}}\bar{p}_{2}^{t_{2}}...\bar{p}%
_{k}^{t_{k}}\right\rangle $ be an ideal of $%
\mathbb{Z}
_{n}$ distinct with $S_{p_{1}}$. Then there exists $j\neq1$ such that
$t_{j}\geq1$, say, $j=k$. Thus, $\bar{p}_{k}^{t_{k}}(\bar{p}_{1}^{t_{1}}%
\bar{p}_{2}^{t_{2}}...\bar{p}_{k-1}^{t_{k-1}})\in I$ but $s\bar{p}_{k}^{t_{k}%
}\notin\sqrt{0}$ and $s(\bar{p}_{1}^{t_{1}}\bar{p}_{2}^{t_{2}}...\bar{p}%
_{k-1}^{t_{k-1}})\notin I$ for all $s\in S_{p_{1}}$. Therefore, $I$ is not an
$S_{p_{1}}$-$n$-ideal of $%
\mathbb{Z}
_{n}$. Similarly, $I$ is not an $S_{p_{i}}$-$n$-ideal of $%
\mathbb{Z}
_{n}$ for all $i=1,2,...,k$.
\end{proof}

\begin{corollary}
Let $n\in%
\mathbb{N}
$. Then for any prime $p$ dividing $n$, either $%
\mathbb{Z}
_{n}$ has no $S_{p}$-$n$-ideals or every every ideal of $%
\mathbb{Z}
_{n}$ disjoint with $S_{p}$ is an $S_{p}$-$n$-ideal.
\end{corollary}

In general if $n=p_{1}^{r_{1}}p_{2}^{r_{2}}...p_{k}^{r_{k}}$ where $r_{i}%
\geq1$ for all $i$, then
\[
S_{p_{1}p_{2}...p_{i-1}p_{i+1}...p_{k}}=\left\{  \bar{p}_{1}^{m_{1}}\bar
{p}_{2}^{m_{2}}...\bar{p}_{i-1}^{m_{i-1}}\bar{p}_{i+1}^{m_{i+1}}...\bar{p}%
_{k}^{m_{k}}:m_{j}\in%
\mathbb{N}
\cup\left\{  0\right\}  \right\}
\]
is also a multiplicatively closed subset of $%
\mathbb{Z}
_{n}$ for all $i$. Next, we generalize Theorem \ref{Zn}.

\begin{theorem}
\label{Zn,gen}Let $n=p_{1}^{r_{1}}p_{2}^{r_{2}}...p_{k}^{r_{k}}$ where
$p_{1},p_{2},...,p_{k}$ are distinct prime integers and $r_{i}\geq1$ for all
$i$.
\end{theorem}

\begin{enumerate}
\item $%
\mathbb{Z}
_{n}$ has no $S_{p_{1}p_{2}...p_{k}}$-$n$-ideals.

\item For $i=1,2,...,k$, every ideal of $%
\mathbb{Z}
_{n}$ disjoint with $S_{p_{1}p_{2}...p_{i-1}p_{i+1}...p_{k}}$ is an
$S_{p_{1}p_{2}...p_{i-1}p_{i+1}...p_{k}}$-$n$-ideal.

\item Let $k\geq3$. If $m\leq k-2$, then $%
\mathbb{Z}
_{n}$ has no $S_{p_{1}p_{2}...p_{m}}$-$n$-ideals.
\end{enumerate}

\begin{proof}
(1) This is clear since $I\cap S_{p_{1}p_{2}...p_{k}}\neq\phi$ for any ideal
$I$ of $%
\mathbb{Z}
_{n}$.

(2) With no loss of generality, we may choose $i=k$. Let $I=\left\langle
\bar{p}_{1}^{t_{1}}\bar{p}_{2}^{t_{2}}...\bar{p}_{k}^{t_{k}}\right\rangle $ be
an ideal of $%
\mathbb{Z}
_{n}$ disjoint with $S_{p_{1}p_{2}...p_{k-1}}$. Then we must have $t_{k}\geq
1$. Choose $s=\bar{p}_{1}^{t_{1}}\bar{p}_{2}^{t_{2}}...\bar{p}_{k-1}^{t_{k-1}%
}\in S_{p_{1}p_{2}...p_{k-1}}$ and let $a,b\in%
\mathbb{Z}
_{n}$ such that $ab\in I$. If $a\in\left\langle \bar{p}_{k}\right\rangle $,
then $sa\in\left\langle \bar{p}_{1}\bar{p}_{2}...\bar{p}_{k}\right\rangle
=\sqrt{0}$. If $a\notin\left\langle \bar{p}_{k}\right\rangle $, then we must
have $b\in\left\langle \bar{p}_{k}^{t_{k}}\right\rangle $. Thus, $sb\in I$ and
$I$ is an $S_{p_{1}p_{2}...p_{k-1}}$-$n$-ideal of $%
\mathbb{Z}
_{n}$.

(3) Assume $m=k-2$ and let $I=\left\langle \bar{p}_{1}^{t_{1}}\bar{p}%
_{2}^{t_{2}}...\bar{p}_{k}^{t_{k}}\right\rangle $ be an ideal of $%
\mathbb{Z}
_{n}$ disjoint with $S_{p_{1}p_{2}...p_{k-2}}$. Then at least one of $t_{k-1}$
and $t_{k}$ is nonzero, say, $t_{k}\gneqq0$. Hence, $\bar{p}_{k}^{t_{k}}%
(\bar{p}_{1}^{t_{1}}\bar{p}_{2}^{t_{2}}...\bar{p}_{k-1}^{t_{k-1}})\in I$ but
clearly $s\bar{p}_{k}^{t_{k}}\notin\sqrt{0}$ and $s(\bar{p}_{1}^{t_{1}}\bar
{p}_{2}^{t_{2}}...\bar{p}_{k-1}^{t_{k-1}})\notin I$ for all $s\in
S_{p_{1}p_{2}...p_{k-2}}$. Therefore, $%
\mathbb{Z}
_{n}$ has no $S_{p_{1}p_{2}...p_{k-2}}$-$n$-ideals. A similar proof can be
used if $1\leq m\lneqq k-2$.
\end{proof}

An ideal $I$ of a ring $R$ is called a maximal $S$-$n$-ideal if there is no
$S$-$n$-ideal of $R$ that contains $I$ properly. In the following proposition,
we observe the relationship between maximal $S$-$n$-ideals and $S$-prime ideals.

\begin{proposition}
\label{max}Let $S\subseteq reg(R)$ be a multiplicatively closed subset of a
ring $R$. If $I$ is a maximal $S$-$n$-ideal of $R$, then $I$ is $S$-prime (and
so $(I:s)=\sqrt{0}$ for some $s\in S$).
\end{proposition}

\begin{proof}
Suppose $I$ is a maximal $S$-$n$-ideal of $R$ and $s\in S$ is an $S$-element
of $I$. Then $(I:s)$ is an $n$-ideal of $R$ by Proposition \ref{(I:s)}.
Moreover, $(I:s)$ is a maximal $n$-ideal of $R$. Indeed, if $(I:s)\subsetneq
J$ for some $n$-ideal (and so $S$-$n$-ideal) $J$ of $R$, then $I\subseteq
(I:s)\subsetneq J$ which is a contradiction. By \cite[Theorem 2.11]{Tekir},
$(I:s)=\sqrt{0}$ is a prime ideal of $R$ and so $I$ is an $S$-prime ideal by
\cite[Proposition 1]{S-prime}.
\end{proof}

\begin{proposition}
Let $S$ be a multiplicatively closed subset of a ring $R$ and $I$ be an ideal
of $R$ disjoint with $S$. If $I$ is an $S$-$n$-ideal, and $J$ is an ideal of
$R$ with $J\cap S\neq\emptyset$, then $IJ$ and $I\cap J~$are $S$-$n$-ideals of
$R$.
\end{proposition}

\begin{proof}
Let $s^{\prime}\in J\cap S$. Let $a,b\in R$ with $ab\in IJ.$ Since $ab\in I$,
we have $sa\in\sqrt{0}$ or $sb\in I$ where $s$ is an $S$-element of $I$.
Hence, $(s^{\prime}s)a\in J\sqrt{0}\subseteq\sqrt{0}$ or $(s^{\prime}s)b\in
IJ$. Thus, $IJ$ is an $S$-$n$-ideal of $R.$ The proof that $I\cap J~$is an
$S$-$n$-ideal is similar.
\end{proof}

\begin{proposition}
Let $S$ be a multiplicatively closed subset of a ring $R$ and $\left\{
I_{\alpha}:\alpha\in\Lambda\right\}  $ be a family of proper ideals of $R$.
\end{proposition}

\begin{enumerate}
\item If $I_{\alpha}$ is an $S$-$n$-ideal of $R$ for all $\alpha\in\Lambda$,
then $\bigcap\limits_{\alpha\in\Lambda}I_{\alpha}$ and is an $S$-$n$-ideal of
$R$.

\item If $\left(  \bigcap\limits_{\alpha\in\Omega}I_{\alpha}\right)  \cap
S\neq\emptyset$ for $\Omega\subseteq\Lambda$ and $I_{\alpha}$ is an $S$%
-$n$-ideal of $R$ for all $\alpha\in\Lambda-\Omega$, then $\bigcap
\limits_{\alpha\in\Lambda}I_{\alpha}$ is an $S$-$n$-ideal of $R$.
\end{enumerate}

\begin{proof}
(1) Suppose that for all $\alpha\in\Lambda$, $I_{\alpha}$ is an $S$-$n$-ideal
of $R$ and note that $\left(  \bigcap\limits_{\alpha\in\Lambda}I_{\alpha
}\right)  \cap S=\phi$. For all $\alpha\in\Lambda$, choose $s_{\alpha}\in S$
such that whenever $a,b\in R$ such that $ab\in I_{\alpha}$, then $s_{\alpha
}a\in\sqrt{0}$ or $s_{\alpha}b\in I_{\alpha}$. Let $a,b\in R$ such that
$ab\in\bigcap\limits_{\alpha\in\Lambda}I_{\alpha}$. Then $ab\in I_{\alpha}$
for all $\alpha\in\Lambda$. If we let $s=\prod\limits_{\alpha\in\Lambda
}s_{\alpha}\in S$, then clearly $sa\in\sqrt{0}$ or $sb\in\bigcap
\limits_{\alpha\in\Lambda}I_{\alpha}$ and the result follows.

(2) Choose $s^{\prime}\in\left(  \bigcap\limits_{\alpha\in\Omega}I_{\alpha
}\right)  \cap S$. Let $a,b\in R$ with $ab\in\bigcap\limits_{\alpha\in\Lambda
}I_{\alpha}$. Then for all $\alpha\in\Lambda-\Omega$, $ab\in I_{\alpha}$ and
so $s_{\alpha}a\in\sqrt{0}$ or $s_{\alpha}b\in I_{\alpha}$ for some
$S$-element $s_{\alpha}$ of $I_{\alpha}$. Hence, $(s^{\prime}\prod
\limits_{\alpha\in\Lambda-\Omega}s_{\alpha})a\in\sqrt{0}$ or $(s^{\prime}%
\prod\limits_{\alpha\in\Lambda-\Omega}s_{\alpha})b\in\bigcap\limits_{\alpha
\in\Lambda}I_{\alpha}$ and so $\bigcap\limits_{\alpha\in\Lambda}I_{\alpha}$ is
an $S$-$n$-ideal of $R$.
\end{proof}

Let $S$ and $T$ be two multiplicatively closed subsets of a ring $R$ with
$S\subseteq T$. Let $I$ be an ideal disjoint with $T.$ It is clear that if $I$
is a $S$-$n$-ideal, then it is $T$-$n$-ideal. The converse is not true since
while $I=<\overline{4}>$ is an $S$-$n$-ideal of $%
\mathbb{Z}
_{12}$ for $S=\{\overline{1},\overline{3},\overline{9}\}$, it is not a $T$%
-$n$-ideal for $T=\{\overline{1}\}\subseteq S$.

\begin{proposition}
\label{subset}Let $S$ and $T$ be two multiplicatively closed subsets of a ring
$R$ with $S\subseteq T$ such that for each $t\in T$, there is an element
$t^{\prime}\in T$ such that $tt^{\prime}\in S$. If $I$ is a $T$-$n$-ideal of
$R$, then $I$ is an $S$-$n$-ideal of $R$.
\end{proposition}

\begin{proof}
Suppose $ab\in I$. Then there is a $T$-element $t\in T$ of $I$ satisfying
$ta\in\sqrt{0}$ or $tb\in I.$ Hence there exists some $t^{\prime}\in T$ with
$s=tt^{\prime}\in S$, and thus $sa\in\sqrt{0}$ or $sb\in I.$
\end{proof}

Let $S$ be a multiplicatively closed subset of a ring $R$. The saturation of
$S$ is the set $S^{\ast}=\{r\in R:\frac{r}{1}$ is a unit in $S^{-1}R\}$. It is
clear that $S^{\ast}$ is a multiplicatively closed subset of $R$ and that
$S\subseteq S^{\ast}$. Moreover, it is well known that $S^{\ast}=\{x\in
R:xy\in S$ for some $y\in R$\}, see \cite{Gilmer}. The set $S$ is called
saturated if $S^{\ast}=S$.

\begin{proposition}
Let $S$ be a multiplicatively closed subset of a ring $R$ and $I$ be an ideal
of $R$ disjoint with $S$. Then $I$ is an $S$-$n$-ideal of $R$ if and only if
$I$ is an $S^{\ast}$-$n$-ideal of $R$.
\end{proposition}

\begin{proof}
Suppose $I$ is an $S^{\ast}$-$n$-ideal of $R$. It is enough by Proposition
\ref{subset} to prove that for each $t\in S^{\ast}$, there is an element
$t^{\prime}\in S^{\ast}$ such that $tt^{\prime}\in S$. Let $t\in S^{\ast}$ and
choose $t^{\prime}\in R$ such that $ty\in S$. Then $t^{\prime}\in S^{\ast}$
and $tt^{\prime}\in S$ as required. The converse is obvious.
\end{proof}

Let $S$ and $T$ be multiplicatively closed subsets of a ring $R$ with
$S\subseteq T$. Then clearly, $T^{-1}S=\left\{  \frac{s}{t}:t\in T\text{,
}s\in S\right\}  $ is a multiplicatively closed subset of $T^{-1}R$.

\begin{proposition}
\label{Inverse g}Let $S$, $T$ be multiplicatively closed subsets of a ring $R$
with $S\subseteq T$ and $I$ be an ideal of $R$ disjoint with $T$. If $I$ is an
$S$-$n$-ideal of $R$, then $T^{-1}I$ is an $T^{-1}S$-$n$-ideal of $T^{-1}R$.
Moreover, we have $T^{-1}I\cap R=(I:u)$ for some $S$-element $u$ of $I$.
\end{proposition}

\begin{proof}
Suppose $I$ is an $S$-$n$-ideal. Suppose $T^{-1}S\cap T^{-1}I\neq\phi$, say,
$\frac{a}{t}\in T^{-1}S\cap T^{-1}I$. Then $a\in S$ and $ta\in I$ for some
$t\in T$. Since $S\subseteq T$, then $ta\in T\cap I$, a contradiction. Thus,
$T^{-1}I$ is proper in $T^{-1}R$ and $T^{-1}S\cap T^{-1}I=\phi$. Let $s\in S$
be an $S$-element of $I$ and choose $\frac{s}{1}\in T^{-1}S$. Suppose $a,b\in
R$ and $t_{1},t_{2}\in T$ with $\frac{a}{t_{1}}\frac{b}{t_{2}}\in T^{-1}I$ and
$\frac{s}{1}\frac{a}{t_{1}}\notin\sqrt{0_{T^{-1}R}}$. Then $tab\in I$ for some
$t\in T$ and $sa\notin\sqrt{0}$. Since $I$ is an $S$-$n$-ideal, we must have
$stb\in I$. Thus, $\frac{s}{1}\frac{b}{t_{2}}=\frac{stb}{tt_{2}}\in T^{-1}I$
as needed. Now, let $r\in T^{-1}I\cap R$ and choose $i\in I$, $t\in T$ such
that $\frac{r}{1}=\frac{i}{t}$. Then $vr\in I$ for some $v\in T$. Since $I$ is
an $S$-$n$-ideal, then there exists $u\in S\subseteq T$ such that $uv\in
\sqrt{0}$ or $ur\in I$. But $uv\notin\sqrt{0}$ as $T\cap\sqrt{0}=\phi$ and so
$ur\in I$. It follows that $r\in(I:u)$ for some $S$-element $u$ of $I$. Since
clearly $(I:u)\subseteq T^{-1}I\cap R$ for all $u\in T$, the proof is completed.
\end{proof}

In particular, if $S=T$, then all elements of $T^{-1}S$ are units in $T^{-1}%
R$. As a special case of of Proposition \ref{Inverse g}, we have the following.

\begin{corollary}
\label{Inverse}Let $S$ be a multiplicatively closed subset of a ring $R$ and
$I$ be an ideal of $R$ disjoint with $S$. If $I$ is an $S$-$n$-ideal of $R$,
then $S^{-1}I$ is an $n$-ideal of $S^{-1}R$. Moreover, we have $S^{-1}I\cap
R=(I:s)$ for some $S$-element $s$ of $I$.
\end{corollary}

\begin{proof}
Suppose $I$ is an $S$-$n$-ideal. Then $S^{-1}I$ is an $S^{-1}S$-$n$-ideal of
$S^{-1}R$ by Proposition \ref{Inverse g}. Let $a,b\in R,$ $s_{1},s_{2}\in S$
with $\frac{a}{s_{1}}\frac{b}{s_{2}}\in S^{-1}I$. Then by assumption,
$\frac{s}{t}\frac{a}{s_{1}}\in\sqrt{0_{S^{-1}R}}$ or $\frac{s}{t}\frac
{b}{s_{2}}\in S^{-1}I$ for some $S^{-1}S$-element $\frac{s}{t}$ of $S^{-1}I$.
Since $\frac{s}{t}$ is a unit in $S^{-1}R$, then $S^{-1}I$ is an $n$-ideal of
$S^{-1}R$ as required. The other part follows directly by Proposition
\ref{Inverse g}.
\end{proof}

\begin{corollary}
Let $S$ be a multiplicatively closed subset of a ring $R$ and $I$ be an ideal
of $R$ disjoint with $S$. Then $I$ is an $S$-$n$-ideal of $R$ if and only if
$S^{-1}I$ is an $n$-ideal of $S^{-1}R$, $S^{-1}I\cap R=(I:s)$ and $S^{-1}%
\sqrt{0}\cap R=(\sqrt{0}:t)$ for some $s,t\in S$.
\end{corollary}

\begin{proof}
$\Rightarrow)$ Suppose $I$ is an $S$-$n$-ideal of $R$. Then $S^{-1}I$ is an
$n$-ideal of $S^{-1}R$ by Corollary \ref{Inverse}. The other part of the
implication follows by using a similar approach to that used in the proof of
Proposition \ref{Inverse g}.

$\Leftarrow)$ Suppose $S^{-1}I$ is an $n$-ideal of $S^{-1}R$, $S^{-1}I\cap
R=(I:s)$ and $S^{-1}\sqrt{0}\cap R=(\sqrt{0}:t)$ for some $s,t\in S$. Choose
$u=st\in S$ and let $a,b\in R$ such that $ab\in I$. Then $\frac{a}{1}\frac
{b}{1}\in S^{-1}I$ and so $\frac{a}{1}\in\sqrt{S^{-1}0}=S^{-1}\sqrt{0}$ or
$\frac{b}{1}\in S^{-1}I$ . If $\frac{a}{1}\in\sqrt{S^{-1}0}$, then there is
$w\in S$ such that $wa\in\sqrt{0}$. Thus, $a=\frac{wa}{w}\in S^{-1}\sqrt
{0}\cap R=(\sqrt{0}:t)$. Hence, $ta\in\sqrt{0}$ and so $ua=sta\in\sqrt{0}$. If
$\frac{b}{1}\in{S^{-1}I}$, then there is $v\in S$ such that $vb\in I$ and so
$b=\frac{vb}{v}\in S^{-1}I\cap R=(I:s)$. Therefore, $ub=tsb\in I$ and $I$ is
an $S$-$n$-ideal of $R$.
\end{proof}

\begin{proposition}
\label{f}Let $f:R_{1}\rightarrow R_{2}$ be a ring homomorphism and $S$ be a
multiplicatively closed subset of $R_{1}$. Then the following statements hold.
\end{proposition}

\begin{enumerate}
\item If $f$ is an epimorphism and $I$ is an $S$-$n$-ideal of $R_{1}$
containing $Ker(f)$, then $f(I)$ is an $f(S)$-$n$-ideal of $R_{2}.$

\item If $Ker(f)\subseteq\sqrt{0_{R_{1}}}$ and $J$ is an $f(S)$-$n$-ideal of
$R_{2}$, then $f^{-1}(J)$ is an $S$-$n$-ideal of $R_{1}.$
\end{enumerate}

\begin{proof}
First we show that $f(I)\cap f(S)=\emptyset$. Otherwise, there is $t\in
f(I)\cap f(S)$ which implies $t=f(x)=f(s)$ for some $x\in I$ and $s\in S$.
Hence, $x-s\in Ker(f)\subseteq I$ and $s\in I$, a contradiction.

(1) Let $a,b\in R_{2}$ and $ab\in f(I)$. Since $f$ is onto, $a=f(x)$ and
$b=f(y)$ for some $x,y\in R_{1}$. Since $f(x)f(y)\in f(I)$ and
$Ker(f)\subseteq I$, we have $xy\in I$ and so there exists an $s\in S$ such
that $sx\in\sqrt{0_{R_{1}}}$ or $sy\in I$. Thus, $f(s)a\in\sqrt{0_{R_{2}}}$ or
$f(s)b\in f(I)$, as needed.

(2) Let $a,b\in R_{1}$ with $ab\in f^{-1}(J).$ Then $f(ab)=f(a)f(b)\in J$ and
since $J$ is an $f(S)$-$n$-ideal of $R_{2},$ there exists $f(s)\in f(S)$ such
that $f(s)f(a)\in\sqrt{0_{R_{2}}}$ or $f(s)f(b)\in J$. Thus, $sa\in
\sqrt{0_{R_{1}}}$ (as $Ker(f)\subseteq\sqrt{0_{R_{1}}}$) or $sb\in f^{-1}(J)$.
\end{proof}

Let $S$ be a multiplicatively closed subset of a ring $R$ and $I$ be an ideal
of $R$ disjoint with $S$. If we denote $r+I\in R/I$ by $\bar{r}$, then clearly
the set $\bar{S}=\{\overline{s}:s\in S\}$ is a multiplicatively closed subset
of $R/I$. In view of Proposition \ref{f}, we conclude the following result for
$\overline{S}$-$n$-ideals of $R/I.$

\begin{corollary}
\label{quot}Let $S$ be a multiplicatively closed subset of a ring $R$ and $I$,
$J$ are two ideals of $R$ with $I\subseteq J$ .
\end{corollary}

\begin{enumerate}
\item If $J$ is an $S$-$n$-ideal of $R$, then $J/I$ is an $\overline{S}$%
-$n$-ideal of $R/I$. Moreover, the converse is true if $I\subseteq\sqrt{0}$.

\item If $R$ is a subring of $R^{\prime}$ and $I^{\prime}$ is an $S$-$n$-ideal
of $R^{\prime}$, then $I^{\prime}\cap R$ is an $S$-$n$-ideal of $R.$
\end{enumerate}

\begin{proof}
(1) Note that $(J/I)\cap\overline{S}=\phi$ if and only if $I\cap S=\phi$. Now,
we apply the canonical epimorphism $\pi:R\rightarrow R/I$ in Proposition
\ref{f}.

(2) Apply the natural injection $i:R\rightarrow R^{\prime}$ in Proposition
\ref{f} (2).
\end{proof}

We recall that a proper ideal $I$ of a ring $R$ is called superfluous if
whenever $I+J=R$ for some ideal $J$ of $R$, then $J=R$.

\begin{proposition}
\label{Sum}Let $S\subseteq reg(R)$ be a multiplicatively closed subset of a
ring $R$.
\end{proposition}

\begin{enumerate}
\item If $I$ is an $S$-$n$-ideal of $R$, then it is superfluous.

\item If $I$ and $J$ are $S$-$n$-ideals of $R$, then $I+J$ is an $S$-$n$-ideal.
\end{enumerate}

\begin{proof}
(1) Suppose $I+J=R$ for some ideal $J$ of $R$ and let $j\in J$. Then $1-j\in
I\subseteq\sqrt{0}\subseteq J(R)$ by (1) of Proposition \ref{p1}. Thus, $j\in
U(R)$ and $J=R$ as needed.

(2) Suppose $I$ and $J$ are $S$-$n$-ideals of $R$. Since $I,J\subseteq\sqrt
{0}$, then $I+J\subseteq\sqrt{0}$ and so $(I+J)\cap S=\phi$. Now, $I/(I\cap
J)$ is an $\overline{S}_{1}$-$n$-ideal of $R/(I\cap J)$ by (1) of Corollary
\ref{quot} where $\bar{S}_{1}=\{s+(I\cap J):s\in S\}$. If $\bar{S}%
_{2}=\{s+J:s\in S\}$, then clearly $\bar{S}_{1}\subseteq\bar{S}_{2}$ and so
$I/(I\cap J)$ is also an $\overline{S}_{2}$-$n$-ideal of $R/(I\cap J)$. By the
isomorphism $(I+J)/J\cong I/(I\cap J)$, we conclude that $(I+J)/J$ is an
$\overline{S}_{2}$-$n$-ideal of $R/J$. Now, the result follows again by (1) of
Corollary \ref{quot}.
\end{proof}

\begin{proposition}
\label{cart}Let $R$ and $R^{\prime}$ be two rings, $I\unlhd R$ and $I^{\prime
}\unlhd R^{\prime}$. If $S$ and $S^{\prime}$ are multiplicatively closed
subsets of $R$ and $R^{\prime}$, respectively, then
\end{proposition}

\begin{enumerate}
\item $I\times R^{\prime}$ is an $(S\times S^{\prime})$-$n$-ideal of $R\times
R^{\prime}$ if and only if $I$ is an $S$-$n$-ideal of $R$ and $S^{\prime}%
\cap\sqrt{0_{R^{\prime}}}\neq\phi$.

\item $R\times I^{\prime}$ is an $(S\times S^{\prime})$-$n$-ideal of $R\times
R^{\prime}$ if and only if $I^{\prime}$ is an $S^{\prime}$-$n$-ideal of
$R^{\prime}$ and $S\cap\sqrt{0_{R}}\neq\phi$.
\end{enumerate}

\begin{proof}
It is clear that $(I\times R^{\prime})\cap(S\times S^{\prime})=\emptyset$ if
and only if $I\cap S=\emptyset$ and $(R\times I^{\prime})\cap(S\times
S^{\prime})=\emptyset$ if and only if $I^{\prime}\cap S^{\prime}=\emptyset$.

(1) Let $a,b\in R$ with $ab\in I$. Choose an $(S\times S^{\prime})$-element
$(s,s^{\prime})$ of $I\times R^{\prime}$. If $sb\notin I$, then $(a,1)(b,1)\in
I\times R^{\prime}$ with $(s,s^{\prime})(b,1)\notin I\times R^{\prime}$. Since
$I\times R^{\prime}$ is an $(S\times S^{\prime})$-$n$-ideal, then
$(s,s^{\prime})(a,1)\in\sqrt{0_{R\times R^{\prime}}}=\sqrt{0_{R}}\times
\sqrt{0_{R^{\prime}}}$. Thus, $sa\in\sqrt{0_{R}}$ and $s^{\prime}\in
S^{\prime}\cap\sqrt{0_{R^{\prime}}}$ $I$. If $sb\in I$, then
$(b,1)(s,s^{\prime})\in I\times R^{\prime}$and so $(s,s^{\prime})(b,1)\in
\sqrt{0_{R\times R^{\prime}}}=\sqrt{0_{R}}\times\sqrt{0_{R^{\prime}}}$ as
$(s,s^{\prime})^{2}\notin I\times R^{\prime}$. In both cases, we conclude that
$I$ is an $S$-$n$-ideal of $R$ and $S^{\prime}\cap\sqrt{0_{R^{\prime}}}%
\neq\phi$. Conversely, suppose $I$ is an $S$-$n$-ideal of $R$, $s$ is some
$S$-element of $I$ and $s^{\prime}\in S^{\prime}\cap\sqrt{0_{R^{\prime}}}$.
Let $(a,a^{\prime})(b,b^{\prime})\in I\times R^{\prime}$ for $(a,a^{\prime
}),(b,b^{\prime})\in R\times R^{\prime}$. Then $ab\in I$ which implies
$sa\in\sqrt{0_{R}}$ or $sb\in I$. Hence, we have either $(s,s^{\prime
})(a,a^{\prime})\in\sqrt{0_{R}}\times\sqrt{0_{R^{\prime}}}$ or $(s,s^{\prime
})(b,b^{\prime})\in I\times R^{\prime}$. Therefore, $(s,s^{\prime})$ is an
$S\times S^{\prime}$-element of $I\times R^{\prime}$ as needed.

(2) Similar to (1).
\end{proof}

The assumptions $S^{\prime}\cap\sqrt{0_{R^{\prime}}}\neq\phi$ and $S\cap
\sqrt{0_{R}}\neq\phi$ in Proposition \ref{cart} are crucial. Indeed, let
$R=R^{\prime}=%
\mathbb{Z}
_{12}$, $S=S^{\prime}=\{\overline{1},\overline{3},\overline{9}\}$ and
$I=<\overline{4}>.$ It is shown in Example \ref{e1} that $I$ is an $S$%
-$n$-ideal of $R$ while $I\times R^{\prime}$ is not an $(S\times S^{\prime}%
)$-$n$-ideal of $R\times R^{\prime}$ as $(\overline{2},\overline{1}%
)(\overline{2},\overline{1})\in I\times R^{\prime}$ but for all $(s,s^{\prime
})\in S\times S$, neither $(s,s^{\prime})(\overline{2},\overline{1})\in
I\times R^{\prime}$ nor $(s,s^{\prime})(\overline{2},\overline{1})\in
\sqrt{0_{R\times R^{\prime}}}$.

\begin{remark}
Let $S$ and $S^{\prime}$ be multiplicatively closed subsets of the rings $R$
and $R^{\prime}$, respectively. If $I$ and $I^{\prime}$ are proper ideals of
$R$ and $R^{\prime}$ disjoint with $S,$ $S^{\prime}$, respectively, then
$I\times I^{\prime}$ is not an $(S\times S^{\prime})$-$n$-ideal of $R\times
R^{\prime}$.
\end{remark}

\begin{proof}
First, note that $S\cap\sqrt{0_{R}}=S^{\prime}\cap\sqrt{0_{R^{\prime}}%
}=\emptyset$. Assume on the contrary that $I\times I^{\prime}$ is an $(S\times
S^{\prime})$-$n$-ideal of $R\times R^{\prime}$ and $(s,s^{\prime})$ is an
$(S\times S^{\prime})$-element of $I\times I^{\prime}.$ Since $(1,0)(0,1)\in
I\times I^{\prime}$, we conclude either $(s,s^{\prime})(1,0)\in\sqrt{0_{R}%
}\times\sqrt{0_{R^{\prime}}}$ or $(s,s^{\prime})(0,1)\in I\times I^{\prime}$
which implies $s\in\sqrt{0_{R}}$ or $s^{\prime}\in I^{\prime}$, a contradiction.
\end{proof}

\begin{proposition}
Let $R$ and $R^{\prime}$ be two rings, $S$ and $S^{\prime}$ be
multiplicatively closed subsets of $R$ and $R^{\prime}$, respectively. If $I$
and $I^{\prime}$ are proper ideals of $R,$ $R^{\prime}$, respectively then
$I\times I^{\prime}$ is an $(S\times S^{\prime})$-$n$-ideal of $R\times
R^{\prime}$ if one of the following statements holds.
\end{proposition}

\begin{enumerate}
\item $I$ is an $S$-$n$-ideal of $R$ and $S^{\prime}\cap\sqrt{0_{R^{\prime}}%
}\neq\phi$.

\item $I^{\prime}$ is an $S^{\prime}$-$n$-ideal of $R^{\prime}$ and
$S\cap\sqrt{0_{R}}\neq\phi$.
\end{enumerate}

\begin{proof}
Clearly $(I\times I^{\prime})\cap(S\times S^{\prime})=\emptyset$ if and only
if $I\cap S=\emptyset$ or $I^{\prime}\cap S^{\prime}=\emptyset.$ Suppose $I$
is an $S$-$n$-ideal of $R$ and $S^{\prime}\cap\sqrt{0_{R^{\prime}}}\neq\phi$.
Then $I\cap S=\emptyset$ and $0_{R^{\prime}}\in I^{\prime}\cap S^{\prime}%
\neq\emptyset$. Choose an $S$-element $s$ of $I$ and let $(a,a^{\prime
})(b,b^{\prime})\in I\times I^{\prime}$ for $(a,a^{\prime}),(b,b^{\prime})\in
R\times R^{\prime}$. Then $ab\in I$ which implies $sa\in\sqrt{0_{R}}$ or
$sb\in I$. Hence, we have either $(s,0)(a,a^{\prime})\in\sqrt{0_{R}}%
\times\sqrt{0_{R^{\prime}}}$ or $(s,0)(b,b^{\prime})\in I\times I^{\prime}$.
Therefore, $(s,0)$ is an $S\times S^{\prime}$-element of $I\times I^{\prime}$.
Similarly, if $I^{\prime}$ is an $S^{\prime}$-$n$-ideal of $R^{\prime}$ and
$S\cap\sqrt{0_{R}}\neq\phi$, then also $I\times I^{\prime}$ is an $(S\times
S^{\prime})$-$n$-ideal of $R\times R^{\prime}$.
\end{proof}

\section{$S$-$n$-ideals of Idealizations and Amalgamations}

Recall that the idealization of an $R$-module $M$ denoted by $R(+)M$ is the
commutative ring $R\times M$ with coordinate-wise addition and multiplication
defined as $(r_{1},m_{1})(r_{2},m_{2})=(r_{1}r_{2},r_{1}m_{2}+r_{2}m_{1})$.
For an ideal $I$ of $R$ and a submodule $N$ of $M$, $I(+)N$ is an ideal of
$R(+)M$ if and only if $IM\subseteq N$. It is well known that if $I(+)N$ is an
ideal of $R(+)M$, then $\sqrt{I(+)N}=\sqrt{I}(+)M$ and in particular,
$\sqrt{0_{R(+)M}}=\sqrt{0}(+)M$. If $S$ is a multiplicatively closed subset of
$R$, then clearly the sets $S(+)M=\left\{  (s,m):s\in S\text{, }m\in
M\right\}  $ and $S(+)0=\left\{  (s,0):s\in S\right\}  $ are multiplicatively
closed subsets of the ring $R(+)M$.

Next, we determine the relation between $S$-$n$-ideals of $R$ and $S(+)M$%
-$n$-ideals of the $R(+)M$.

\begin{proposition}
\label{Idealiz}Let $N$ be a submodule of an $R$-module $M$, $S$ be a
multiplicatively closed subset of $R$ and $I$ be an ideal of $R$ where
$IM\subseteq N$. If $I(+)N$ is an $S(+)M$-$n$-ideal of $R(+)M$, then $I$ is an
$S$-$n$-ideal of $R$.
\end{proposition}

\begin{proof}
Clearly, $S\cap I=\phi$. Choose an $S(+)M$-element $(s,m)$ of $I(+)N$ and let
$a,b\in R$ such that $ab\in I$. Then $(a,0)(b,0)\in I(+)N$ and so
$(s,m)(a,0)\in\sqrt{0}(+)M$ or $(s,m)(b,0)\in I(+)N$. Hence, $sa\in\sqrt{0}$
or $sb\in I$ and $I$ is an $S$-$n$-ideal of $R$
\end{proof}

\begin{proposition}
\label{Id}Let $S$ be a multiplicatively closed subset of a ring $R$, $I$ be an
ideal of $R$ disjoint with $S$ and $M$ be an $R$-module. The following are equivalent.
\end{proposition}

\begin{enumerate}
\item $I$ is an $S$-$n$-ideal of $R$.

\item $I(+)M$ is an $S(+)0$-$n$-ideal of $R(+)M$.

\item $I(+)M$ is an $S(+)M$-$n$-ideal of $R(+)M$.
\end{enumerate}

\begin{proof}
(1)$\Rightarrow$(2). Suppose $I$ is an $S$-$n$-ideal of $R$, $s$ is an
$S$-element of $I$ and note that $S(+)0\cap$ $I(+)M=\phi$. Choose $(s,0)\in
S(+)0$ and let $(a,m_{1}),(b,m_{2})\in R(+)M$ such that $(a,m_{1})(b,m_{2})\in
I(+)M$. Then $ab\in I$ and so either $sa\in\sqrt{0}$ or $sb\in I$. It follows
that $(s,0)(a,m_{1})\in\sqrt{0}(+)M=\sqrt{0_{R(+)M}}$ or $(s,0)(b,m_{2})\in
I(+)M$. Thus, $I(+)M$ is an $S(+)0$-$n$-ideal of $R(+)M$.

(2)$\Rightarrow$(3). Clear since $S(+)0\subseteq S(+)M$.

(3)$\Rightarrow$(1). Proposition \ref{Idealiz}.
\end{proof}

\begin{remark}
The converse of Proposition \ref{Idealiz} is not true in general. For example,
if $S=\left\{  1,-1\right\}  $, then $0$ is an $S$-$n$-ideal of $%
\mathbb{Z}
$ but $0(+)\bar{0}$ is not an $(S(+)%
\mathbb{Z}
_{6})$-$n$-ideal of $%
\mathbb{Z}
(+)%
\mathbb{Z}
_{6}$. For example, $(2,\bar{0})(0,\bar{3})\in0(+)\bar{0}$ but clearly
$(s,m)(2,\bar{0})\notin\sqrt{0}(+)%
\mathbb{Z}
_{6}=\sqrt{0_{%
\mathbb{Z}
(+)%
\mathbb{Z}
_{6}}}$ and $(s,m)(0,\bar{3})\notin0(+)\bar{0}$ for all $(s,m)\in S(+)%
\mathbb{Z}
_{6}$.
\end{remark}

Let $R$ and $R^{\prime}$ be two rings, $J$ be an ideal of $R^{\prime}$ and
$f:R\rightarrow R^{\prime}$ be a ring homomorphism. The set $R\Join
^{f}J=\left\{  (r,f(r)+j):r\in R\text{, }j\in J\right\}  $ is a subring of
$R\times R^{\prime}$ called the amalgamation of $R$ and $R^{\prime}$ along $J$
with respect to $f$. In particular, if $Id_{R}:R\rightarrow R$ is the identity
homomorphism on $R$, then $R\Join J=R\Join^{Id_{R}}J=\left\{  (r,r+j):r\in
R\text{, }j\in J\right\}  $ is the amalgamated duplication of a ring along an
ideal $J$. Many properties of this ring have been investigated and analyzed
over the last two decades, see for example \cite{DAnna2}, \cite{DAnna3}.

Let $I$ be an ideal of $R$ and $K$ be an ideal of $f(R)+J$. Then $I\Join
^{f}J=\left\{  (i,f(i)+j):i\in I\text{, }j\in J\right\}  $ and $\bar{K}%
^{f}=\{(a,f(a)+j):a\in R$, $j\in J$, $f(a)+j\in K\}$ are ideals of $R\Join
^{f}J$, \cite{DAnna3}. For a multiplicatively closed subset $S$ of $R$, one
can easily verify that $S\Join^{f}J=\left\{  (s,f(s)+j):s\in S\text{, }j\in
J\right\}  $ and $W=\left\{  (s,f(s)):s\in S\right\}  $ are multiplicatively
closed subsets of $R\Join^{f}J$. If $J\subseteq\sqrt{0_{R^{\prime}}}$, then
one can easily see that $\sqrt{0_{R\Join^{f}J}}=\sqrt{0_{R}}\Join^{f}J$.

Next, we determine when the ideal $I\Join^{f}J$ is $(S\Join^{f}J)$-$n$-ideal
in $R\Join^{f}J$.

\begin{theorem}
\label{ama}Consider the amalgamation of rings $R$ and $R^{\prime}$ along the
ideals $J$ of $R^{\prime}$ with respect to a homomorphism $f$. Let $S$ be a
multiplicatively closed subset of $R$ and $I$ be an ideal of $R$ disjoint with
$S$. Consider the following statements:

(1) $I\Join^{f}J$ is a $W$-$n$-ideal of $R\Join^{f}J$.

(2) $I\Join^{f}J$ is a $(S\Join^{f}J)$-$n$-ideal of $R\Join^{f}J$.

(3) $I$ is a $S$-$n$-ideal of $R$.

Then (1) $\Rightarrow$ (2) $\Rightarrow$ (3). Moreover, if $J\subseteq
\sqrt{0_{R^{\prime}}}$, then the statements are equivalent.
\end{theorem}

\begin{proof}
(1)$\Rightarrow$(2). Clear as $W\subseteq S\Join^{f}J$.

(2)$\Rightarrow$(3). First note that $(S\Join^{f}J)\cap(I\Join^{f}%
J)=\emptyset$ if and only if $S\cap I=\emptyset$. Suppose $I\Join^{f}J$ is an
$(S\Join^{f}J)$-$n$-ideal of $R\Join^{f}J$. Choose an $(S\Join^{f}J)$-element
$(s,f(s))$ of $I\Join^{f}J$. Let $a,b\in R$ such that $ab\in I$ and
$sa\notin\sqrt{0_{R}}$. Then $(a,f(a))(b,f(b))\in I\Join^{f}J$ and clearly
$(s,f(s))(a,f(a))\notin\sqrt{0_{R\Join^{f}J}}$. Hence, $(s,f(s))(b,f(b))\in
I\Join^{f}J$ and so $sb\in I$. Thus, $s$ is an $S$-element of $I$ and $I$ is
an $S$-$n$-ideal of $R$.

Now, suppose $J\subseteq\sqrt{0_{R^{\prime}}}$. We prove (3)$\Rightarrow$(1).
Suppose $s$ is an $S$-element of $I$ and let $(a,f(a)+j_{1})(b,f(b)+j_{2}%
)=(ab,(f(a)+j_{1})(f(b)+j_{2}))\in I\Join^{f}J$ for $(a,f(a)+j_{1}%
),(b,f(b)+j_{1})\in R\Join^{f}J$. If $(s,f(s))(a,f(a)+j_{1})\notin
\sqrt{0_{R\Join^{f}J}}=\sqrt{0_{R}}\Join^{f}J$, then $sa\notin\sqrt{0_{R}}$.
Since $ab\in I$, we conclude that $sb\in I$ and so $(s,f(s))(b,f(b)+j_{2})\in
I\Join^{f}J$. Thus, $(s,f(s))$ is a $W$-element of $I\Join^{f}J$ and
$I\Join^{f}J$ is a $W$-$n$-ideal of $R\Join^{f}J$.
\end{proof}

\begin{corollary}
\label{cj}Consider the amalgamation of rings $R$ and $R^{\prime}$ along the
ideal $J\subseteq\sqrt{0_{R^{\prime}}}$ of $R^{\prime}$ with respect to a
homomorphism $f$. Let $S$ be a multiplicatively closed subset of $R$. The
$(S\Join^{f}J)$-$n$-ideals of $R\Join^{f}J$ containing $\{0\}\times J$ are of
the form $I\Join^{f}J$ where $I$ is a $S$--$n$-ideal of $R.$
\end{corollary}

\begin{proof}
From Theorem \ref{ama}, $I\Join^{f}J$ is a $(S\Join^{f}J)$-$n$-ideal of
$R\Join^{f}J$ for any $S$-$n$-ideal $I$ of $R$. Let $K$ be a $(S\Join^{f}%
J)$-$n$-ideal of $R\Join^{f}J$ containing $\{0\}\times J.$ Consider the
surjective homomorphism $\varphi:R\Join^{f}J\rightarrow R$ defined by
$\varphi(a,f(a)+j)=a$ for all $(a,f(a)+j)\in R\Join^{f}J$. Since
$Ker(\varphi)=\{0\}\times J\subseteq K$, then $I:=\varphi(K)$ is a $S$%
-$n$-ideal of $R$ by Proposition \ref{f}. Since $\{0\}\times J\subseteq K$, we
conclude that $K=I\Join^{f}J$.
\end{proof}

Let $T$ be a multiplicatively closed subset of $R^{\prime}$. Then clearly, the
set $\bar{T}^{f}=\{(s,f(s)+j):s\in R,$ $j\in J,$ $f(s)+j\in T\}$ is a
multiplicatively closed subset of $R\Join^{f}J$.

\begin{theorem}
\label{ama2}Consider the amalgamation of rings $R$ and $R^{\prime}$ along the
ideals $J$ of $R^{\prime}$ with respect to an epimorphism $f$. Let $K$ be an
ideal of $R^{\prime}$ and $T$ be a multiplicatively closed subset of
$R^{\prime}$ disjoint with $K$. If $\bar{K}^{f}$ is a $\bar{T}^{f}$-$n$-ideal
of $R\Join^{f}J$, then $K$ is a $T$-$n$-ideal of $R^{\prime}$. The converse is
true if $J\subseteq\sqrt{0_{R^{\prime}}}$ and $Ker(f)\subseteq\sqrt{0_{R}}$.
\end{theorem}

\begin{proof}
First, note that $T\cap K=\phi$ if and only if $\bar{T}^{f}\cap\bar{K}%
^{f}=\phi$. Suppose $\bar{K}^{f}$ is a $\bar{T}^{f}$-$n$-ideal of $R\Join
^{f}J$ and $(s,f(s)+j)$ is some $\bar{T}^{f}$-element of $\bar{K}^{f}$. Let
$a^{\prime}$,$b^{\prime}\in R^{\prime}$ such that $a^{\prime}b^{\prime}\in K$
and choose $a,b\in R$ where $f(a)=a^{\prime}$ and $b=f(b^{\prime})$. Then
$(a,f(a)),(b,f(b))\in R\Join^{f}J$ with $(a,f(a))(b,f(b))=(ab,f(ab))\in\bar
{K}^{f}$. By assumption, we have either
$(s,f(s)+j)(a,f(a))=(sa,(f(s)+j)f(a))\in\sqrt{0_{R\Join^{f}J}}$ or
$(s,f(s)+j)(b,f(b))=(sb,(f(s)+j)f(b))\in\bar{K}^{f}$. Thus, $f(s)+j\in T$ and
clearly, $(f(s)+j)f(a)\in\sqrt{0_{R^{\prime}}}$ or $(f(s)+j)f(b)\in K$. It
follows that $K$ is a $T$-$n$-ideal of $R^{\prime}$. Now, suppose $K$ is a
$T$-$n$-ideal of $R^{\prime}$, $t=f(s)$ is a $T$-element of $K$,
$J\subseteq\sqrt{0_{R^{\prime}}}$ and $Ker(f)\subseteq\sqrt{0_{R}}$. Let
$(a,f(a)+j_{1})(b,f(b)+j_{2})=(ab,(f(a)+j_{1})(f(b)+j_{2}))\in\bar{K}^{f}$ for
$(a,f(a)+j_{1}),(b,f(b)+j_{2})\in R\Join^{f}J$. Then $(f(a)+j_{1}%
)(f(b)+j_{2})\in K$ and so $f(s)(f(a)+j_{1})\in\sqrt{0_{R^{\prime}}}$ or
$f(s)(f(b)+j_{2})\in K$. Suppose $f(s)(f(a)+j_{1})\in\sqrt{0_{R^{\prime}}}$.
Since $J\subseteq\sqrt{0_{R^{\prime}}}$, then $f(sa)\in\sqrt{0_{R^{\prime}}}$
and so $(sa)^{m}\in Ker(f)\subseteq\sqrt{0_{R}}$ for some integer $m$. Hence,
$sa\in\sqrt{0_{R}}$ and $(s,f(s))(a,f(a)+j_{1})\in\sqrt{0_{R\Join^{f}J}}$. If
$f(s)(f(b)+j_{2})\in K$, then clearly, $(s,f(s))(b,f(b)+j_{2})\in\bar{K}^{f}$.
Therefore, $\bar{K}^{f}$ is a $\bar{T}^{f}$-$n$-ideal of $R\Join^{f}J$ as needed.
\end{proof}

In particular, $S\times f(S)$ is a multiplicatively closed subset of
$R\Join^{f}J$ for any multiplicatively closed subset $S$ of $R$. Hence, we
have the following corollary of Theorem \ref{ama2}.

\begin{corollary}
\label{ama3}Let $R$, $R^{\prime}$, $J$, $S$ and $f$ be as in Theorem
\ref{ama}. Let $K$ be an ideal of $R^{\prime}$ and $T=f(S)$. Consider the
following statements.

(1) $\bar{K}^{f}$ is a $(S\times T)$-$n$-ideal of $R\Join^{f}J$.

(2) $\bar{K}^{f}$ is a $\bar{T}^{f}$-$n$-ideal of $R\Join^{f}J$.

(3) $K$ is a $T$-$n$-ideal of $R$.

Then (1) $\Rightarrow$ (2) $\Rightarrow$ (3). Moreover, if $J\subseteq
\sqrt{0_{R^{\prime}}}$ and $Ker(f)\subseteq\sqrt{0_{R}}$, then the statements
are equivalent.
\end{corollary}

We note that if $J\nsubseteq\sqrt{0_{R^{\prime}}}$, then the equivalences in
Theorems \ref{ama} and \ref{ama2} are not true in general.

\begin{example}
Let $R=%
\mathbb{Z}
$, $I=\left\langle 0\right\rangle =K$, $J=\left\langle 3\right\rangle
\nsubseteq\sqrt{0_{%
\mathbb{Z}
}}$ and $S=\left\{  1\right\}  =T$. We have $I\Join J=\left\{  (0,3n):n\in%
\mathbb{Z}
\right\}  $, $\bar{K}=\left\{  (3n,0):n\in%
\mathbb{Z}
\right\}  $, $S\Join J=\left\{  (1,3n+1):n\in%
\mathbb{Z}
\right\}  $, $\bar{T}=\left\{  (1-3n,1):n\in%
\mathbb{Z}
\right\}  $ and $\sqrt{0_{R\Join J}}=\left\{  (0,0)\right\}  $.
\end{example}

\begin{enumerate}
\item $I$ is a $S$-$n$-ideal of $R$ but $I\Join J$ is not a $(S\Join J)$%
-$n$-ideal of $R\Join J$. Indeed, we have $(0,3),(1,4)\in R\Join J$ with
$(0,3)(1,4)=(0,12)\in I\Join J$. But $(1,3n+1)(0,3)\notin\sqrt{0_{R\Join J}}$
and $(1,3n+1)(1,4)\notin I\Join J$ for all $n\in%
\mathbb{Z}
$.

\item $K$ is a $T$-$n$-ideal of $R$ but $\bar{K}$ is not a $\bar{T}$-$n$-ideal
of $R\Join J$. For example, $(-3,0),(-4,-1)\in R\Join J$ with
$(-3,0)(-4,-1)=(12,0)\in\bar{K}$. However, $(1-3n,1)(-3,0)\notin
\sqrt{0_{R\Join J}}$ and $(1-3n,1)(-4,-1)\notin\bar{K}$ for all $n\in%
\mathbb{Z}
$.
\end{enumerate}

By taking $S=\left\{  1\right\}  $ in Theorem \ref{ama} and Corollary
\ref{ama3}, we get the following particular case.

\begin{corollary}
Let $R$, $R^{\prime}$, $J$, $I$, $K$ and $f$ be as in Theorems \ref{ama} and
\ref{ama2}.
\end{corollary}

\begin{enumerate}
\item If $I\Join^{f}J$ is an $n$-ideal of $R\Join^{f}J$, then $I$ is an
$n$-ideal of $R$. Moreover, the converse is true if $J\subseteq\sqrt
{0_{R^{\prime}}}$.

\item If $\bar{K}^{f}$ is an $n$-ideal of $R\Join^{f}J$, then $K$ is an
$n$-ideal of $R^{\prime}$. Moreover, the converse is true if $J\subseteq
\sqrt{0_{R^{\prime}}}$ and $Ker(f)\subseteq\sqrt{0_{R}}$.
\end{enumerate}

\begin{corollary}
\label{11}Let $R$, $R^{\prime}$,$I$, $J$, $K$, $S$ and $T$ be as in Theorems
\ref{ama} and \ref{ama2}.
\end{corollary}

\begin{enumerate}
\item If $I\Join J$ is a $(S\Join J)$-$n$-ideal of $R\Join J$, then $I$ is a
$S$--$n$-ideal of $R$. Moreover, the converse is true if $J\subseteq
\sqrt{0_{R^{\prime}}}$.

\item If $\bar{K}$ is a $\bar{T}$-$n$-ideal of $R\Join J$, then $K$ is a
$T$-$n$-ideal of $R^{\prime}$. The converse is true if $J\subseteq
\sqrt{0_{R^{\prime}}}$ and $Ker(f)\subseteq\sqrt{0_{R}}$.
\end{enumerate}

As a generalization of $S$-$n$-ideals to modules, in the following we define
the notion of $S$ -$n$-submodules which may inspire the reader for the other work.

\begin{definition}
Let $S$ be a multiplicatively closed subset of a ring $R$, and let $M$ be a
unital $R$-module. A submodule $N$ of $M$ with $(N:_{R}M)\cap S$ $=\emptyset$
is called an $S$ -$n$-submodule if there is an $s\in S$ such that $am\in N$
implies $sa\in\sqrt{(0:_{R}M)}$ or $sm\in N$ for all $a\in R$ and $m\in M.$
\end{definition}

\end{document}